\documentclass{amsart}
\usepackage{amsmath}
\usepackage{amssymb}
\usepackage{amsthm}
\usepackage{epsfig}
\usepackage{amsfonts}
\usepackage{amscd}
\usepackage{mathrsfs}
\usepackage{enumerate}
\usepackage{latexsym}

\newtheorem{theorem}{Theorem}[section]
\newtheorem{lemma}[theorem]{Lemma}

\newtheorem{proposition}[theorem]{Proposition}

\theoremstyle{definition}
\theoremstyle{notation}

\newtheorem{definition}[theorem]{Definition}

\theoremstyle{remark}

\numberwithin{equation}{section}

\begin{document}

\title[A pseudocompactification]{A pseudocompactification}

\author[M.R. Koushesh]{M.R. Koushesh}

\address{Department of Mathematical Sciences, Isfahan University of Technology, Isfahan 84156--83111, Iran}

\address{School of Mathematics, Institute for Research in Fundamental Sciences (IPM), P.O. Box: 19395--5746, Tehran, Iran.}

\email{koushesh@cc.iut.ac.ir}



\thanks{This research  was in part supported by a grant from IPM (No. 90030052).}

\begin{abstract}
For a locally pseudocompact space $X$ let
\[\zeta X=X\cup\mbox{cl}_{\beta X}(\beta X\backslash\upsilon X).\]
It is proved that $\zeta X$ is the largest (with respect to the standard partial order $\leq$) among all pseudocompactifications of $X$ which have compact remainder. Other characterizations of $\zeta X$ are also given.
\end{abstract}

\maketitle

\section{Introduction}

A space $Y$ is called an {\em extension} of a space $X$ if $Y$ contains $X$ as a dense subspace. If $Y$ is an extension of $X$ then the subspace $Y\backslash X$ of $Y$ is called the {\em remainder} of $Y$. Two extensions of $X$ are said to be {\em equivalent} if there exists a homeomorphism between them which fixes $X$ pointwise. This defines an equivalence relation on the class of all extensions of $X$. The equivalence classes will be identified with individuals. Pseudocompact extensions are called {\em pseudocompactifications}.

Let $X$ be a Tychonoff non-pseudocompact space. In \cite{AS}, C.E. Aull and J.O. Sawyer have considered the pseudocompactification
\[\alpha X=X\cup(\beta X\backslash\upsilon X)\]
of $X$ and they have studied its characterizations among pseudocompactifications of $X$ contained in $\beta X$. Specifically, they have proved that $\alpha X$ is the smallest pseudocompactification $Y$ of $X$ contained in $\beta X$ such that every free hyper-real $z$-ultrafilter in $X$ converges in $Y$, and $\alpha X$ is the largest pseudocompactification $Y$ of $X$ contained in $\beta X$ such that every point of $Y\backslash X$ is contained in a zero-set of $Y$ which misses $X$. (A $z$-ultrafilter in $X$ is said to be {\em real} if it has the countable intersection property; otherwise, it is called {\em hyper-real}.)

Here in this note, motivated by the results of \cite{AS} and our previous work \cite{Ko3} (also \cite{Ko4}), for a Tychonoff space $X$ we consider the subspace
\[\zeta X=X\cup\mbox{cl}_{\beta X}(\beta X\backslash\upsilon X)\]
of $\beta X$. We show that if $X$ is locally pseudocompact, $\zeta X$ is the largest (with respect to the standard partial order $\leq$) of all pseudocompactifications of $X$ which have compact remainder. We give other characterizations of $\zeta X$, including a characterization of $\zeta X$ via $z$-ultrafilter in $X$.

We denote by ${\mathscr Z}(X)$ and $Coz(X)$ the set of all zero-sets and the set of all cozero-sets of a space $X$, respectively. As usual, we denote by $\beta X$ and $\upsilon X$ the Stone-\v{C}ech compactification and the Hewitt realcompactification of a space $X$, respectively. We refer to \cite{E}, \cite{GJ}, \cite{PW} and \cite{W} for undefined terms and notation and background materials.

\section {The definition of $\zeta X$}

In this section we formally define $\zeta X$ and consider the cases when it takes on the familiar forms $\zeta X=X$, $\zeta X=\beta X$ and $\zeta X=\alpha X$.

\begin{definition}\label{TR}
For a Tychonoff space $X$ let
\[\zeta X=X\cup\mbox{\em cl}_{\beta X}(\beta X\backslash\upsilon X)=X\cup(\beta X\backslash\mbox{\em int}_{\beta X}\upsilon X)\]
considered as a subspace of $\beta X$.
\end{definition}

Note that $\zeta X$ is always pseudocompact, as it densely contains the pseudocompactification $\alpha X$.

The following result is due to  A.W. Hager and D.G. Johnson in \cite{HJ}; a direct proof may be found in \cite{C}. (See also Theorem 11.24 of \cite{W}.)

\begin{lemma}[Hager-Johnson \cite{HJ}]\label{A}
Let $U$ be an open subset of the Tychonoff space $X$. If $\mbox{\em cl}_{\upsilon X} U$ is compact then $\mbox{\em cl}_X U$ is pseudocompact.
\end{lemma}

\begin{lemma}\label{HGA}
Let $A$ be a regular closed subset of the Tychonoff space $X$. Then $\mbox{\em cl}_{\beta X} A\subseteq\upsilon X$ if and only if $A$ is pseudocompact.
\end{lemma}

\begin{proof}
The first half follows from Lemma \ref{A}. For the second half, note that if $A$ is pseudocompact then so is $\mbox{cl}_{\upsilon X} A$. But  $\mbox{cl}_{\upsilon X} A$, being closed in $\upsilon X$, is also realcompact, and thus compact. Therefore $\mbox{cl}_{\beta X} A\subseteq\mbox{cl}_{\upsilon X} A$.
\end{proof}

For an open subset $U$ of the Tychonoff space $X$ denote
\[\mbox{Ex}_XU=\beta X\backslash\mbox {cl}_{\beta X}(X\backslash U).\]
Note that $\mbox{Ex}_XU$  is open in $\beta X$ and $X\cap\mbox{Ex}_XU=U$.

The following lemma is motivated by Lemma 2.17 of \cite{Ko4}.

\begin{lemma}\label{JIUH}
Let $X$ be a Tychonoff  space. Then
\[\mbox{\em int}_{\beta X}\upsilon X=\bigcup\big\{\mbox{\em Ex}_X C:C\in Coz(X)\mbox{ and }\mbox{\em cl}_X C \mbox{ is pseudocompact}\big\}.\]
\end{lemma}

\begin{proof}
If $C\in Coz(X)$ has pseudocompact closure then $\mbox{cl}_{\beta X} C\subseteq\upsilon X$, by Lemma \ref{HGA}. But then $\mbox{Ex}_X C\subseteq \mbox{int}_{\beta X}\upsilon X$, as $\mbox{Ex}_X C\subseteq\mbox{cl}_{\beta X} C$.

For the reverse inclusion, let $t\in\mbox{int}_{\beta X}\upsilon X$. Let $f:\beta X\rightarrow[0,1]$ be continuous with $f(t)=0$ and $f|(\beta X\backslash\mbox{int}_{\beta X}\upsilon X)\equiv 1$. Then $C=X\cap f^{-1}[[0,1/2)]$ is a cozero-set of $X$ with $t\in\mbox{Ex}_X C$. Also, $\mbox{cl}_X C$ is pseudocompact, by Lemma \ref{HGA}.
\end{proof}

\begin{lemma}\label{BA27}
Let $X$ be a Tychonoff space and let $Z\in{\mathscr Z}(X)$. Then $\mbox{\em cl}_{\beta X}Z\subseteq\mbox{\em int}_{\beta X}\upsilon X$ if and only if $Z$ is contained in a cozero-set of $X$ with pseudocompact closure.
\end{lemma}

\begin{proof}
If $Z\subseteq C$, where $C\in Coz(X)$ has pseudocompact closure, then
\[\mbox{cl}_{\beta X}Z\cap\mbox{cl}_{\beta X}(X\backslash C)=\emptyset\]
as $Z$ and $X\backslash C$ are disjoint zero-sets of $X$. Thus, by Lemma \ref{JIUH}
\[\mbox{cl}_{\beta X}Z\subseteq\beta X\backslash\mbox {cl}_{\beta X}(X\backslash C)=\mbox{Ex}_X C\subseteq\mbox{int}_{\beta X}\upsilon X.\]

For the converse, if $\mbox{cl}_{\beta X }Z\subseteq\mbox{int}_{\beta X}\upsilon X$, then
\[\mbox{cl}_{\beta X }Z\subseteq\mbox{Ex}_X C_1\cup\cdots\cup \mbox{Ex}_X C_n\]
where each $C_1,\ldots,C_n\in Coz(X)$ has pseudocompact closure. If $C=C_1\cup\cdots\cup C_n$, then $Z\subseteq C$, $C\in Coz(X)$ and $\mbox{cl}_X C$ is pseudocompact.
\end{proof}

Recall that in the $z$-ultrafilter representation of $\beta X$ the points of $\upsilon X$ correspond to those $z$-ultrafilters in $X$ which have the countable intersection property.

\begin{theorem}\label{JGD}
Let $X$ be a Tychonoff space. Then
\begin{itemize}
\item[\rm(1)] $\zeta X=X$ if and only if $X$ is pseudocompact.
\item[\rm(2)] $\zeta X=\beta X$ if and only if there exists no cozero-set of $X$ with pseudocompact closure containing a non-compact zero-set of $X$.
\item[\rm(3)] The following are equivalent:
\begin{itemize}
\item[\rm(a)] $\zeta X=\alpha X$.
\item[\rm(b)] $\upsilon X\backslash X\subseteq\mbox{\em int}_{\beta X}\upsilon X$.
\item[\rm(c)] Every free $z$-ultrafilter in $X$ with the countable intersection property has an element contained in a cozero-set of $X$ with pseudocompact closure.
\end{itemize}
\end{itemize}
\end{theorem}

\begin{proof}
(1). Obviously, if $\zeta X=X$ then $X$ is pseudocompact. For the converse, if $X$ is pseudocompact, then $\zeta X=X$, as $\upsilon X=\beta X$.

(2). Note that $\zeta X\neq\beta X$ if and only if $\mbox{int}_{\beta X}\upsilon X\backslash X\neq\emptyset$ if and only if there exists a non-compact  $Z\in{\mathscr Z}(X)$ with $\mbox{cl}_{\beta X}Z\subseteq\mbox{int}_{\beta X}\upsilon X$. The result now follows from Lemma \ref{BA27}.

(3). The equivalence of (3.a) and (3.b) is obvious. (3.b) {\em implies} (3.c). Let $\mathscr{F}$ be a free $z$-ultrafilter in $X$ with the countable intersection property. Then
\[\bigcap_{Z\in\mathscr{F}}\mbox{cl}_{\beta X}Z\in\mbox{int}_{\beta X}\upsilon X\]
and thus
\[\bigcap_{Z\in\mathscr{F}}\mbox{cl}_{\beta X}Z\in\mbox{Ex}_X C\]
by Lemma \ref{JIUH}, for some $C\in Coz(X)$ with pseudocompact closure. Therefore
\[\mbox{cl}_{\beta X}Z_1\cap\cdots\cap\mbox{cl}_{\beta X}Z_n\subseteq\mbox{Ex}_X C\]
for some $Z_1,\ldots,Z_n\in\mathscr{F}$. If $Z=Z_1\cap\cdots\cap Z_n$, then $Z\in\mathscr{F}$ and $Z\subseteq C$.

(3.c) {\em implies} (3.b). Let $t\in\upsilon X\backslash X$. Then
\[t=\bigcap_{Z\in\mathscr{F}}\mbox{cl}_{\beta X}Z\]
for some free $z$-ultrafilter $\mathscr{F}$ in $X$ with the countable intersection property. Let $C$ be a cozero-set of $X$ with pseudocompact closure, containing an element $Z$ of $\mathscr{F}$. Then
\[t\in\mbox{cl}_{\beta X}Z\subseteq\mbox{Ex}_X C\subseteq\mbox{int}_{\beta X}\upsilon X\]
by the proof of Lemma \ref{BA27}.
\end{proof}

In \cite{C0}, W.W. Comfort describes a locally compact space $X$ such that $\upsilon X$ is not locally compact. For this space $X$ we necessarily have $\zeta X\neq\alpha X$; as otherwise, by the above theorem we have $\upsilon X\backslash X\subseteq\mbox{int}_{\beta X}\upsilon X$ and thus (since $X\subseteq\mbox{int}_{\beta X}\upsilon X$, as $X$ being locally compact, is open in $\beta X$) $\upsilon X\subseteq\mbox{int}_{\beta X}\upsilon X$, which is not possible.

We conclude this section with a result which characterizes spaces $X$ with locally compact Hewitt realcompactification $\upsilon X$. The proof is as of  the one given for Theorem \ref{JGD} above; the result, however, may also be deduced from W.W. Comfort's result in \cite{C}. (See \cite{Har} for an alternative characterization of such spaces $X$, and \cite{Hag}, for a characterization of spaces $X$ with locally compact $\sigma$-compact $\upsilon X$.)

\begin{proposition}\label{JGF}
For a  Tychonoff space $X$ the following are equivalent:
\begin{itemize}
\item[\rm(1)] $\upsilon X$ is locally compact.
\item[\rm(2)] Every $z$-ultrafilter in $X$ with the countable intersection property has an element contained in a cozero-set of $X$ with pseudocompact closure.
\end{itemize}
\end{proposition}

\section {The local pseudocompactness of $X$}

A Tychonoff space $X$ is called {\em locally pseudocompact} if every point of $X$ has an open neighborhood with pseudocompact closure. Local pseudocompactness will be crucial here; we first focus on that.

\begin{definition}\label{UTY}
For a Tychonoff space $X$ let
\[R(X)=\mbox{\em cl}_{\beta X}(\beta X\backslash\upsilon X)=\beta X\backslash\mbox{\em int}_{\beta X}\upsilon X.\]
Therefore
\[\zeta X=X\cup R(X).\]
\end{definition}

\begin{theorem}\label{FGA}
For a Tychonoff space $X$ the following are equivalent:
\begin{itemize}
\item[\rm(1)] $X$ is locally pseudocompact.
\item[\rm(2)] {\em (Comfort \cite{C})} $X\subseteq\mbox{\em int}_{\beta X}\upsilon X$.
\item[\rm(3)] $X$ and $R(X)$ are disjoint.
\item[\rm(4)] $\zeta X\backslash X$ is compact.
\item[\rm(5)] $X$ has a pseudocompactification with compact remainder.
\end{itemize}
\end{theorem}

\begin{proof}
The equivalence of (1) and (2) is due to W.W. Comfort \cite{C}. Obviously, (2) and (3) are equivalent, (3) implies (4), and (4) implies (5). That (5) implies (1) follows from the fact that pseudocompactness is hereditary with respect to regular closed subsets.
\end{proof}

\section {The general form of pseudocompactifications of $X$ with compact remainder}

Our next purpose in this note is to characterize $\zeta X$ among all pseudocompactifications of $X$ with compact remainder. But before we proceed with this, let us find the general form of all such pseudocompactifications. This will be done in this section.

The following lemma is well known; we include the proof here for the sake of completeness.

\begin{lemma}\label{j2}
Let $X$ be a Tychonoff space, let $Y$ be a Tychonoff extension of $X$ with compact remainder and let $\phi:\beta X\rightarrow\beta Y$ continuously extend $\mbox{\em id}_X$. Then $\beta Y$ coincides with the quotient space of $\beta X$ obtained by contracting each fiber $\phi^{-1}(p)$, for $p\in Y\backslash X$, to $p$, and $\phi$ is the quotient mapping.
\end{lemma}

\begin{proof}
Let $Y\backslash X=\{p_i:i\in I\}$ where $p_i$'s are bijectively indexed. Let $T$ be the space obtained from $\beta X$ by contracting each fiber $\phi^{-1}(p_i)$ where $i\in I$ to a point $a_i$. We show that $T=\beta Y$ (identifying each $a_i$ with $p_i$). First, we show that $T$ is a compactification of $Y$. To show that $T$ is Hausdorff let $s,t\in T$ be distinct. Consider the following cases:
\begin{description}
\item[{\sc Case 1}] Suppose that $s,t\in T\backslash \{a_i:i\in I\}$. Then $s,t\in \beta X\backslash\phi^{-1}[Y\backslash X]$ and thus there exist disjoint open neighborhoods $U$ and $V$ of $s$ and $t$ in $\beta X$, respectively, each disjoint from $\phi^{-1}[Y\backslash X]$. The sets $q[U]$ and $q[V]$ are disjoint open neighborhoods of $s$ and $t$ in $T$, respectively.
\item[{\sc Case 2}] Suppose that $s=a_i$ for some $i\in I$ and $t\in T\backslash \{a_i:i\in I\}$. Then $\phi^{-1}[Y\backslash X]$ is a compact subset of $\beta X$ not containing $t$ and thus there exist disjoint open subsets $U$ and $V$ of $\beta X$ such that $\phi^{-1}[Y\backslash X]\subseteq U$ and $t\in V$. Now $q[U]$ and $q[V]$ are disjoint open neighborhoods of $s$ and $t$ in $T$, respectively. The case when  $s\in T\backslash\{a_i:i\in I\}$ and $t=a_j$ for some  $j\in I$ is analogous.
\item[{\sc Case 3}] Suppose that $s=a_i$ and $t=a_j$ for some $i,j\in I$. Let $U_i $ and $U_j$ be disjoint open neighborhoods of $p_i$ and $p_j$  in $\beta Y$, respectively. Since $q^{-1}[q[\phi^{-1}[U_k]]]=\phi^{-1}[U_k]$, where $k=i,j$, are open subsets of $\beta X$ and  $\phi^{-1}(p_k)\subseteq\phi^{-1}[U_k]$ the sets $q[\phi^{-1}[U_k]]$, where $k=i,j$, are disjoint open neighborhoods of $s$ and $t$  in $T$, respectively.
\end{description}
This shows that $T$ is Hausdorff and therefore, being a continuous image of $\beta X$, it is compact. Note that $Y$ is a subspace of $T$. To show this first note that since $\beta Y$ is also a compactification of $X$ we have $\phi[\beta X\backslash X]=\beta Y\backslash X$.  Now if $W$ is open in $\beta Y$, since $q^{-1}[q[\phi^{-1}[W]]]=\phi^{-1}[W]$ is open in $\beta X$ the set $q[\phi^{-1}[W]]$ is open in $T$, and therefore
\[W\cap Y=q\big[\phi^{-1}[W]\big]\cap Y\]
is open in $Y$ as a subspace of $T$. For the converse, note that if $W$ is an open subset of $T$, then
\[W\cap Y=\big(\beta Y\backslash\phi\big[\beta X\backslash q^{-1}[W]\big]\big)\cap Y\]
and therefore (since $\phi[\beta X\backslash q^{-1}[W]]$ is compact) the set $W\cap Y$ is open in $Y$ in its original topology. Clearly, $Y$ is dense in $T$ and therefore $T$ is a compactification of $Y$. To  show that $T=\beta Y$ it suffices to verify that any continuous $f:Y\rightarrow[0,1]$ can be continuously extended over $T$. Indeed, consider the continuous mapping
\[g=fq:S=X\cup\phi^{-1}[Y\backslash X]\rightarrow[0,1].\]
Note that $\beta S=\beta X$. Let $g_\beta:\beta X\rightarrow[0,1]$ be the continuous extension of $g$. Define $F:T\rightarrow[0,1]$ such that $F(x)=g_\beta(x)$ for any $x\in\beta X\backslash\phi^{-1}[Y\backslash X]$ and $F(p_i)=f(p_i)$ for any $i\in I$. Then $F|Y=f$ and since $Fq=g_\beta$ is continuous, $F$ is continuous. This shows that $T=\beta Y$. Note, this also implies that $\phi=q$, as they both coincide on $X$.
\end{proof}

\begin{lemma}\label{ASD}
Let $X$ be a Tychonoff space, let $Y$ be a Tychonoff extension of $X$ with compact remainder, let $K$ be a compactification of $Y$ and let $\phi:\beta X\rightarrow K$ continuously extend $\mbox{\em id}_X$. Then the following are equivalent:
\begin{itemize}
\item[\rm(1)] $Y$ is a pseudocompactification of $X$.
\item[\rm(2)] $\mbox{\em cl}_{\beta X}(\beta X\backslash\upsilon X)\subseteq\phi^{-1}[Y\backslash X]$.
\end{itemize}
\end{lemma}

\begin{proof}
We prove the lemma first in the case when $K=\beta Y$. Note that since $\phi^{-1}[Y\backslash X]$ is closed in $\beta X$, condition (2) is equivalent to the requirement that $\beta X\backslash\upsilon X\subseteq\phi^{-1}[Y\backslash X]$.

(1) {\em implies} (2). Let $x\in\beta X\backslash\upsilon X$  and suppose to the contrary that $x\notin\phi^{-1}[Y\backslash X]$. Let $P\in{\mathscr Z}(\beta X)$ be such that $x\in P$ and $P\cap X=\emptyset$. Now $G=P\backslash\phi^{-1}[Y\backslash X]$ is non-empty (as it contains $x$) and it is a countable intersection of open subsets of $\beta X$ each missing $\phi^{-1}[Y\backslash X]$. Thus (using Lemma \ref{j2}) $G$ is a non-empty $G_\delta$-set of $\beta Y$ which misses $Y$, contradicting the pseudocompactness of $Y$.

(2) {\em implies} (1). Suppose to the contrary that $Y$ is not pseudocompact. Let $p\in \beta Y\backslash\upsilon Y$ and let $Z\in {\mathscr Z}(\beta Y)$ be such that $p\in Z$ and $Z\cap Y=\emptyset$. Then $\phi^{-1}[Z]\in{\mathscr Z}(\beta X)$ misses $X$, and thus
\[\phi^{-1}[Z]\subseteq \beta X\backslash\upsilon X\subseteq\phi^{-1}[Y\backslash X].\]
Since $p\in\phi^{-1}[Z]$ (as $\phi(p)=p$; see Lemma \ref{j2}) we have $p\in\phi^{-1}[Y\backslash X]$ or $p=\phi(p)\in Y\backslash X$, which contradicts the choice of $Z$.

Now suppose that $K$ is an arbitrary compactification of $Y$. Denote by $\psi:\beta X\rightarrow K$ and $\gamma:\beta Y\rightarrow K$ the continuous extensions of $\mbox{id}_X$ and $\mbox{id}_Y$, respectively. Note that $\gamma\psi=\phi$, as they agree on $X$, and $\gamma[\beta Y\backslash Y]=K\backslash Y$. The lemma now follows, as
\[\psi^{-1}[Y\backslash X]=\psi^{-1}\big[\gamma^{-1}[Y\backslash X]\big]=(\gamma\psi)^{-1}[Y\backslash X]=\phi^{-1}[Y\backslash X].\]
\end{proof}

\begin{theorem}\label{IGD}
Let $X$ be a locally pseudocompact space, let $K$ be a compactification of $X$ and let $\phi:\beta X\rightarrow K$ continuously extend $\mbox{\em id}_X$. Then $\phi[\zeta X]$ is the smallest (with respect to $\subseteq$) pseudocompactification of $X$ with compact remainder contained in $K$.
\end{theorem}

\begin{proof}
Note that $\phi[\zeta X]\backslash X=\phi[\zeta X\backslash X]$, as $\phi[\beta X\backslash X]=K\backslash X$. Thus $\phi[\zeta X]$ is an extension of $X$ with compact remainder. That $\phi[\zeta X]$ is pseudocompact follows from Lemma \ref{ASD} (and Theorem \ref{FGA}), as
\[\zeta X\backslash X\subseteq\phi^{-1}\big[\phi[\zeta X\backslash X]\big]=\phi^{-1}\big[\phi[\zeta X]\backslash X\big].\]
If $Y\subseteq K$ is a pseudocompactification of $X$ with compact remainder, then
\[\phi[\zeta X\backslash X]\subseteq\phi\big[\phi^{-1}[Y\backslash X]\big]\subseteq Y\backslash X\]
by Lemma \ref{ASD}, and therefore $\phi[\zeta X]\subseteq Y$.
\end{proof}

\begin{theorem}\label{RES}
Let $X$ be a locally pseudocompact space. Let $K$ be a compactification of $X$, let $\phi:\beta X\rightarrow K$ continuously extend $\mbox{\em id}_X$ and let $E$ be a compact subset of $K\backslash X$ containing $\phi[\zeta X\backslash X]$. Then the subspace $Y=X\cup E$ of $K$ is a pseudocompactification of $X$ with compact remainder. Furthermore, every pseudocompactification of $X$ with compact remainder is of this form.
\end{theorem}

\begin{proof}
This follows from Lemma \ref{ASD} (and Theorem \ref{FGA}).
\end{proof}

\section {External characterization of $\zeta X$}

For two extensions $Y$ and $Y'$ of a space $X$ we let $Y\leq Y'$ if there exists a continuous mapping of $Y'$ into $Y$ which fixes $X$ pointwise. The relation $\leq$ defines a partial order on the set of all (equivalence classes of) extensions of $X$. (See Section 4.1 of \cite{PW} for more details.)

\begin{theorem}\label{PFAJ}
Let $X$ be a locally pseudocompact space. Then
\[\zeta X=\max\big(\{Y:Y\mbox{ is a pseudocompactification of }X\mbox{ with compact remainder}\},\leq\big).\]
\end{theorem}

\begin{proof}
This is now clear. Note that if $Y$ is a pseudocompactification of $X$ with compact remainder and if $\phi:\beta X\rightarrow\beta Y$ denotes the continuous extension of $\mbox{id}_X$, then $\phi[\zeta X]\subseteq Y$, by the proof of Theorem \ref{IGD}. Therefore $\psi=\phi|\zeta X$ continuously maps $\zeta X$ into $Y$, fixing $X$ pointwise. Thus $Y\leq\zeta X$.
\end{proof}

\section {Internal characterization of $\zeta X$}

Let $X$ be a space and let ${\mathscr F}$ a filter-base in $X$. The set
\[\bigcap_{F\in{\mathscr F}}\mbox{cl}_X F\]
is called the {\em adherence} of ${\mathscr F}$. If ${\mathscr F}$ has an empty adherence then it is called {\em free} in $X$.

\begin{theorem}\label{AEA}
Let $X$ be a locally pseudocompact space. Then $\zeta X$ is the unique Tychonoff extension $Y$ of $X$ with compact remainder satisfying the following conditions:
\begin{itemize}
\item[\rm(1)] A free $z$-ultrafilter in $X$ is free in $Y$ if and only if it has an element contained in a cozero-set of $X$ with pseudocompact closure.
\item[\rm(2)] Distinct $z$-ultrafilters in $X$ have disjoint adherences in $Y$.
\end{itemize}
\end{theorem}

\begin{proof}
Note that for a Tychonoff extension $Y$ of $X$ with compact remainder, condition (2) is equivalent to the requirement that $Y\subseteq\beta X$; as (using Lemma \ref{j2} and its notation) this implies that $\phi^{-1}(p)$ is a singleton for each $p\in Y\backslash X$. But if $Y\subseteq\beta X$, then using Lemma \ref{BA27}, the two implication in (1) simply mean $Y\backslash X\subseteq\zeta X\backslash X$ and $\zeta X\backslash X\subseteq Y\backslash X$, and thus $Y=\zeta X$. That $\zeta X$ satisfies (1) is obvious and follows from Lemma \ref{BA27}.
\end{proof}

Our final theorem is analogous to Theorem 4.37 of \cite{Ko3}. Recall that for any Tychonoff space $X$, if $S,Z\in{\mathscr Z}(X)$ then
\[\mbox{cl}_{\beta X} S\cap\mbox{cl}_{\beta X} Z=\mbox{cl}_{\beta X}(S\cap Z).\]

\begin{theorem}\label{OUF}
Let $X$ be a locally pseudocompact space. Then $\zeta X$ is the unique pseudocompactification $Y$ of $X$ with compact remainder satisfying the following condition:
\begin{itemize}
\item[\rm($\star$)] $\mbox{\em cl}_Y S\cap\mbox{\em cl}_Y Z\subseteq X$ for every $S,Z\in{\mathscr Z}(X)$ such that $S\cap Z$ is contained in a cozero-set of $X$ with pseudocompact closure.
\end{itemize}
\end{theorem}

\begin{proof}
That $\zeta X$ satisfies $(\star)$ is obvious, as by Lemma \ref{BA27}, for every $S,Z\in {\mathscr Z}(X)$ such that $S\cap Z\subseteq C$ for some $C\in Coz(X)$ with pseudocompact closure, we have
\[\mbox{cl}_{\beta X} S\cap\mbox{cl}_{\beta X} Z=\mbox{cl}_{\beta X}(S\cap Z)\subseteq\mbox{int}_{\beta X}\upsilon X.\]

Now suppose that $Y$ is a pseudocompactification of $X$ with compact remainder satisfying $(\star)$. We first show that $Y\subseteq\beta X$ by  (using Lemma \ref{j2} and its notation) showing that $\phi^{-1}(p)$ is a singleton for each $p\in Y\backslash X$. But this follows easily, as otherwise, there exist distinct $a,b\in\phi^{-1}(p)$ and disjoint $S,Z\in{\mathscr Z}(X)$ with $a\in \mbox{cl}_{\beta X}S$ and $b\in \mbox{cl}_{\beta X}Z$. But then $p\in\mbox{cl}_Y S\cap\mbox{cl}_Y Z$, which contradicts $(\star)$. (Note that $X$, being locally pseudocompact, contains a cozero-set with pseudocompact closure.) Condition $(\star)$ in particular implies that $\mbox{cl}_{\beta X}S\cap(Y\backslash X)=\emptyset$, for every $S\in{\mathscr Z}(X)$ contained in a cozero-set of $X$ with pseudocompact closure, thus $Y\backslash X\subseteq\mbox{cl}_{\beta X}(\beta X\backslash\upsilon X)$. The reverse inclusion in the latter follows from Lemma \ref{ASD}. Therefore $Y=\zeta X$.
\end{proof}

\noindent{\bf Acknowledgements.} The author wishes to thank the referee for his/her comments, which led to some simplifications of proofs and rephrasements of some results, considerably improving the overall exposition of the article.


\begin{thebibliography}{10}

\bibitem{AS} C.E. Aull and J.O. Sawyer,  The pseudocompact extension $\alpha X$, Proc. Amer. Math. Soc. 102 (1988) 1057--1064.

\bibitem{C0} W.W. Comfort, Locally compact realcompactifications, General Topology and its Relations to Modern Analysis and Algebra II, Proceedings of the Second Prague Topological Symposium, 1966, 95--100.

\bibitem{C} W.W. Comfort, On the Hewitt realcompactification of a product space, Trans. Amer. Math. Soc. 131 (1968) 107--118.

\bibitem{E} R. Engelking, General Topology, second edition, Heldermann Verlag, Berlin, 1989.

\bibitem{GJ} L. Gillman and M. Jerison, Rings of Continuous Functions, Springer-Verlag, New York-Heidelberg, 1976.

\bibitem{Hag} A.W. Hager, On the tensor product of function rings, Doctoral dissertation, Pennsylvania State University, University Park, 1965.

\bibitem{HJ} A.W. Hager and D.G. Johnson, A note on certain subalgebras of $C(X)$, Canad. J. Math. 20 (1968) 389--393.

\bibitem{Har} D. Harris, The local compactness of $\upsilon X$, Pacific J. Math. 50 (1974) 469--476.

\bibitem{Ko4} M.R. Koushesh, The partially ordered set of one-point extensions, Topology Appl. 158 (2011) 509--532.

\bibitem{Ko3} M.R. Koushesh, Compactification-like extensions, Dissertationes Math. (Rozprawy Mat.) 476 (2011) 88 pp.

\bibitem{PW} J.R. Porter and R.G. Woods, Extensions and Absolutes of Hausdorff Spaces, Springer-Verlag, New York, 1988.

\bibitem{W} M.D. Weir,  Hewitt-Nachbin Spaces, American Elsevier, New York, 1975.

\end{thebibliography}
\end{document}